\RequirePackage{fix-cm}
\documentclass[smallextended]{svjour3}       
\smartqed  
\usepackage{graphicx}
\usepackage{amsmath, ntheorem}
\usepackage{amssymb}
\usepackage{stackrel}
\usepackage{marvosym, enumerate}
\usepackage{amstext, amscd, latexsym}
\usepackage{amsfonts, ragged2e}
\usepackage{mathrsfs, pgfplots, enumerate, setspace}
\usepackage{subfigure}
\usepackage{cases}
\usepackage{epstopdf}
\usepackage{booktabs}
\usepackage{tabularx}
\usepackage[ruled,linesnumbered]{algorithm2e}
\usepackage{float} 
\usepackage[font={bf},textfont=md,labelfont={footnotesize},labelformat={default}]{caption}
\numberwithin{equation}{section}
\smartqed
\usepackage{cite}
\usepackage[colorlinks,linkcolor=blue,anchorcolor=blue,citecolor=blue]{hyperref}
\usepackage{amstext, graphicx, amscd, latexsym, amssymb}
\usepackage{amsfonts, ragged2e}
\usepackage{pgfplots, setspace}
\usepackage{geometry}
\usepackage{latexsym,bm}
\usepackage{float}
\usepackage{threeparttable}
\usepackage{array}
\usepackage{booktabs}
\usepackage{makecell}
\usepackage[justification=centering,labelfont=bf]{caption}
\usepackage{etoolbox}
\usepackage[figuresright]{rotating}
\begin{document}

\title{Obtaining properly Pareto optimal solutions of multiobjective optimization problems via the branch and bound method}

\titlerunning{}        

\author{Weitian Wu$^1$ \and Xinmin Yang$^2$ }


\institute{ W.T. Wu \at School of Sciences, Ningbo University of Technology 315211, China\\
                    \email{weitianwu@nbut.edu.cn}\\
           \Letter X.M. Yang \at National Center for Applied Mathematics of Chongqing 401331, China\\
           School of Mathematical Sciences, Chongqing Normal University, Chongqing 401331, China\\
              xmyang@cqnu.edu.cn}

\date{Received: date / Accepted: date}

\maketitle

\begin{abstract}

In multiobjective optimization, most branch and bound algorithms provide the decision maker with the whole Pareto front, and then decision maker could select a single solution finally. However, if the number of objectives is large, the number of candidate solutions may be also large, and it may be difficult for the decision maker to select the most interesting solution. As we argue in this paper, the most interesting solutions are the ones whose trade-offs are bounded. These solutions are usually known as the properly Pareto optimal solutions. We propose a branch-and-bound-based algorithm to provide the decision maker with so-called $\epsilon$-properly Pareto optimal solutions. The discarding test of the algorithm adopts a dominance relation induced by a convex polyhedral cone instead of the common used Pareto dominance relation. In this way, the proposed algorithm excludes the subboxes which do not contain $\epsilon$-properly Pareto optimal solution from further exploration. We establish the global convergence results of the proposed algorithm. Finally, the algorithm is applied to benchmark problems as well as to two real-world optimization problems.

\end{abstract}

\section{Introduction}
Many real-world optimization problems involve multiple objectives that require simultaneous consideration. Researchers refer to such problems as multiobjective optimization problems (MOPs). Since objectives are conflicting, it is impossible to find a single solution that is optimal for all objectives. Instead, a number of Pareto optimal solutions can be identified, characterized by the fact that improving one objective can only be achieved at the expense of worsening at least one other objective. Therefore, none of the Pareto optimal solutions can be said to be inferior when compared to others. The idea of solving an MOP can be understood as helping the decision maker in making a trade-off between multiple objectives and in finding a Pareto optimal solution that pleases him/her the most.

The a posteriori methods attempt to generate a well distributed set of representatives of the whole Pareto optimal solution set. Then the decision maker has to look at this potentially set of alternative solutions and make a choice. Most branch and bound algorithms for MOPs \cite{ref10,ref11,ref26,ref42,ref43,ref55,ref56,ref57} can be classified as a posteriori. However, their solution processes are considered resource-intensive and time-consuming, because the solutions with the required accuracy can only be obtained by continuously subdividing the variable space. Furthermore, if the number of objectives is large, not only the visualization of the high-dimensional Pareto front is not intuitive, but also the number of representative solutions may be huge. In this case, even if a global perspective of the problem is provided, it is very difficult to make a choice for decision maker.

To address the above issue, Wu and Yang \cite{ref44} proposed a reference point-based branch and bound algorithm which can provide the regions of interest that obey the decision maker's preferences, instead of the whole Pareto front. The participation of preferences helps the algorithm to guide the search towards the regions of interest and avoiding computational resources being wasted on exploring undesirable regions. Thus, the algorithm not only requires less computation cost, but also effectively alleviates the decision maker's selection pressure when solving many-objective optimization problems. However, the algorithm requires the assumption that the decision maker has informed preferences, i.e., the decision maker has sufficient a priori knowledge to specify preferences. This assumption does not hold in many situations \cite{ref21}, which makes the algorithm lose its advantages.

We note that some solutions are naturally preferred for the decision maker, even if no informed preference can be proposed. A classical type of the naturally preferred solutions is the properly Pareto optimal solutions. Kuhn and Tucker \cite{ref18} introduced the properly Pareto optimal solutions to avoid providing decision makers with some inappropriate Pareto optimal solutions whose trade-offs do not essentially differ from a weakly Pareto optimal solution. Therefore, from the view of trade-offs, the properly Pareto optimal solutions are more attractive to the decision maker than the improperly ones because their trade-offs are bounded. However, the methods available so far for finding the properly Pareto optimal solutions are almost based on the scalarization methods\cite{ref48,ref49,ref50,ref51,ref52}, so they cannot get rid of some limitations of scalarization methods. For instance, these methods require a set of predetermined parameters to control the distribution of the obtained solutions, however in some problems, choosing almost all parameters leads to unbounded problems \cite{ref53}. Furthermore, it is a frequent observation in some cases that even for convex problems, an evenly distributed set of parameters fails to produce an even distribution of solutions \cite{ref54}.

In this paper, we propose a branch-and-bound-based algorithm to approximate so-called $\epsilon$-properly Pareto optimal solutions \cite{ref38} for nonconvex multiobjective optimization problems whose objective functions have the known Lipschitz constants. The discarding test of the proposed algorithm adopts a new dominance relation induced by a convex polyhedral ordering cone instead of the Pareto dominance relation. We prove the discarding test with the new dominance relation is able to exclude the subboxes which do not contain $\epsilon$-properly Pareto optimal solutions, and further establish that the global convergence of the algorithm. Finally, the proposed algorithm is applied to two- and five-objective engineering constrained optimization problems as well as to benchmark problems. The proposed algorithm has the following capabilities:
\begin{itemize}
  \item Informed preferences of the decision maker are not required;
  \item The algorithm approximates the $\epsilon$-properly Pareto optimal solution set, instead of the Pareto optimal set;
  \item The global convergence of the algorithm can be proved;
  \item The algorithm is applicable to three or more objectives, and linear or nonlinear constraints.
\end{itemize}

The rest of this paper is organized as follows. In Section 2, we introduce the basic concepts and notations of multiobjective optimization. The new branch and bound algorithm and its theoretical analysis is described in Section 3. Section 4 is devoted to some numerical results.

\section{Basics of multiobjective optimization}
In this section we introduce the basic concepts which we need for the new algorithm. A multiobjective optimization problem can be written as follows:
\begin{align}\label{MOP}
\min\limits_{x\in\Omega}\quad F(x)=(f_1(x),\ldots ,f_m(x))^T
\end{align}
with
\begin{align*}
\Omega=\{x\in\mathbb{R}^n:g_j(x)\geq0,\;j=0,\ldots,p,\;\underline{x}_k\leq x_k\leq \overline{x}_k,\;k=0,\ldots,n\},
\end{align*}
where $f_i:\mathbb{R}^n\rightarrow \mathbb{R}$ ($i=1,\ldots,m$) are Lipschitz continuous, and $g_j:\mathbb{R}^n\rightarrow \mathbb{R}$ ($j=0,\ldots,p$) are continuous. If we allow $j=0$, the set $\Omega$ is referred to as a box constraint. In this case, we call $\Omega$ a \emph{box} with the midpoint $m(\Omega)=(\frac{\underline{x}_1+\overline{x}_1}{2},\ldots,\frac{\underline{x}_n+\overline{x}_n}{2})^T$ and the width $\omega(\Omega)=(\overline{x}_1-\underline{x}_1,\ldots,\overline{x}_n-\underline{x}_n)^T$. The diameter of $\Omega$ is denoted by $\|\omega(\Omega)\|$. For a feasible solution $x\in\Omega$, the objective vector $F(x)\in\mathbb{R}^m$ is said to be the image of $x$, while $x$ is called the preimage of $F(x)$.

The concept of Pareto dominance relation between two solutions $x^1,x^2\in\Omega$ can be defined as follows:

\begin{align*}
x^1\;dominates\;x^2~\Longleftrightarrow\;x^1\leq x^2\;\Longleftrightarrow\;F(x^2)-F(x^1)\in \mathbb{R}^m_+\backslash\{0\}.
\end{align*}
We can define similar terms in the objective space. A nonempty set $N\subseteq\mathbb{R}^m$ is called a \emph{nondominated set} if for any $z^1,z^2\in N$ we have $z^1\nleq z^2$ and $z^2\nleq z^1$.

A point $x^*\in\Omega$ is said to be a \emph{Pareto optimal solution} for (MOP) if there does not exist any $x\in\Omega$ such that $F(x)\leq F(x^*)$. The set of all Pareto solutions is called the \emph{Pareto set} and is denoted by $X^*$. The image of Pareto set under the mapping $F$ is called the \emph{Pareto front}.

The aim of an approximation algorithm is used to found an $\varepsilon$-efficient solution of problem (2.1), which is defined next. Let $e$ denote the $m$-dimensional all-ones vector $(1,\ldots,1)^T\in\mathbb{R}^m$.

\begin{definition}\cite{ref47}
Let $\varepsilon\geq0$ be given. A point $\bar{x}\in \Omega$ is an $\varepsilon$-efficient solution of problem (2.1) if there does not exist another $x\in \Omega$ with $F(x)\leq F(\bar{x})-\varepsilon e$.
\end{definition}

We use $d(a,b)=\|a-b\|$ to quantify the distance between two points $a$ and $b$, where $\|\cdot\|$ denotes the Euclidean norm. The distance between the point $a$ and a nonempty finite set $B$ is defined as $d(a,B):=\min_{b\in B}\|a-b\|.$ Let $A$ be another non-empty finite set, we define the Hausdorff distance between $A$ and $B$ by

\begin{align*}
d_H(A,B):=\max\{d_h(A,B),d_h(B,A)\},
\end{align*}
where $d_h(A,B)$ is the directed Hausdorff distance from $A$ to $B$, defined by

\begin{align*}
d_h(A,B):=\max_{a\in A}\{\min_{b\in B}\|a-b\|\}.
\end{align*}

\section{Proposed branch and bound algorithm}
Branch and bound algorithms \cite{ref10,ref11,ref26,ref42,ref43,ref55,ref56,ref57} have been used as the a posteriori methods to solve multiobjective optimization problems. By means of a tree search, a branch and bound algorithm systematically searches for an approximation of the entire Pareto set. The basic branch and bound algorithm for MOPs has been proposed by Fern{\'a}ndez and T{\'o}th \cite{ref11} (see Algorithm 1). The solution process consists of three components:
\begin{itemize}
  \item \emph{branching}: subboxes are bisected perpendicularly to the direction of maximum width;
  \item \emph{bounding}: the lower and upper bounds for subboxes are calculated;
  \item \emph{pruning}: the subboxes that are provably suboptimal are excluded from exploration.
\end{itemize}

\begin{algorithm}[H]\label{alg1}
  \SetKwInOut{Input}{Input}\SetKwInOut{Output}{Output}
  \Input{problem (2.1), termination criterion;}
  \Output{$\mathcal{B}_{k}$, $\mathcal{U}^{nds}$;}
  \BlankLine
  $\mathcal{B}_0\leftarrow \Omega$, $\mathcal{U}^{nds}\leftarrow \emptyset$, $k=0$\;
  \While{termination criterion is not satisfied}
  {$\mathcal{B}_{k+1}\leftarrow \emptyset$\;
  \While{$\mathcal{B}_k\neq\emptyset$}{
  Select $B\in\mathcal{B}_k$ and remove it from $\mathcal{B}_k$\;
  $B_1,B_2\longleftarrow$ Bisect $B$ perpendicularly to the direction of maximum width\;
  \For{$i=1,2$}
  {Calculate the lower bound $l(B_i)$ and upper bound $u(B_i)$ for $B_i$\;
   \If{$B_i$ can not be discarded}{
        Update $\mathcal{U}^{nds}$ by $u(B_i)$ and store $B_i$ into $\mathcal{B}_{k+1}$\;}}
  }
  $k\leftarrow k+1$.}
  \caption{A basic branch and bound algorithm for MOPs}
\end{algorithm}
\bigskip

The source of the upper bounds is usually the image of the midpoints or the vertexes of subboxes. The approaches for the lower bounds proposed so far in the literature include the natural interval extension \cite{ref11,ref25}, the Lipschitz bound \cite{ref42,ref43} and the $\alpha$BB method \cite{ref26}, and the resulting lower bound $l=(l_1,\ldots,l_m)^T$ for a subbox $B$ satisfies the following condition:
\begin{align}
  l\leq F(x),\quad  x\in B.\label{IE:3.1}
\end{align}
Numerical experiments show that there is no big difference among the three bounding approaches. However, the latter two calculate the maximal error between the lower bounds and optimal values.

The pruning can be achieved by \emph{discarding tests}. The discarding test limits the tree search and thus avoids exhaustive enumeration. A common type of discarding test is based on the Pareto dominance relation:

\emph{A subbox will be discarded if there exists a feasible objective vector such that the objective vector dominates the lower bound of the subbox.}

It is easy to see that the subbox is removed because it does not contain any Pareto optimal solutions.

\subsection{The new discarding test}

As discussed earlier, we aim to approximate the properly Pareto optimal solutions. Here we consider the $\epsilon$-properly Pareto optimal solution proposed by Wierzbicki \cite{ref38}:


\begin{definition}\label{de:1}
The solution $x^*\in\Omega$ is said to be the $\epsilon$-properly Pareto optimal solution of problem (2.1), if
\begin{align*}
(F(x^*)-\mathbb{R}^m_{\epsilon}\backslash\{0\})\cap F(\Omega)=\emptyset,
\end{align*}
where $\mathbb{R}^m_{\epsilon}=\{y\in\mathbb{R}^m:\min_{i=1,\ldots,m} (1-\epsilon)y_i+\epsilon\sum_{i=1}^{m}y_i\geq 0\}$, $0\leq\epsilon\leq1$.
\end{definition}

Figure \ref{fig2} depicts the $\epsilon$-properly Pareto optimal solution of the bi-objective optimization problem. The $\epsilon$-properly Pareto optimal solution can be obtained by intersecting the feasible region with a blunt cone $\mathbb{R}^m_{\epsilon}$. Compared to the Pareto optimal solution, the $\epsilon$-properly Pareto optimal solution uses a larger set $\mathbb{R}^m_{\epsilon}$ instead of $\mathbb{R}^m_+$, so the $\epsilon$-properly Pareto optimal solution set is contained in the Pareto optimal solution set. Furthermore, an interesting aspect of $\epsilon$-properly Pareto optimal solutions is that the trade-offs are bounded by $\epsilon$ and $1/\epsilon$ \cite{ref23,ref59}.


\begin{figure}[h]%
\centering
\includegraphics[width=0.5\textwidth]{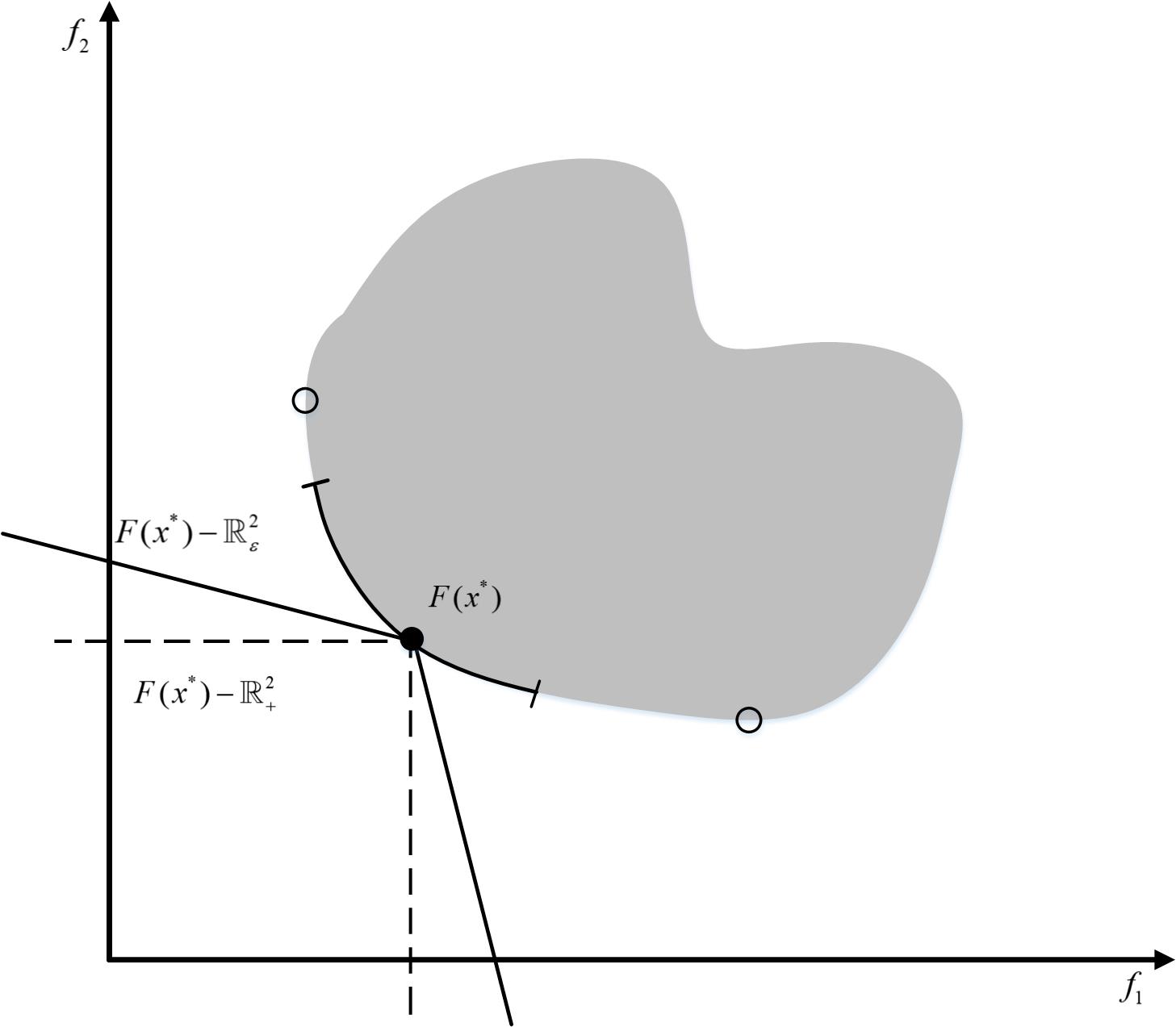}
\caption{$\epsilon$-properly Pareto optimal solution}\label{fig2}
\end{figure}

For $0\leq\epsilon\leq1$, now we define a linear mapping $\mathcal{T}_{\epsilon}:\mathbb{R}^m\rightarrow\mathbb{R}^m$,
\begin{align*}
  \mathcal{T}_{\epsilon}(y):=\begin{bmatrix}
1\quad & \quad \epsilon \quad &\quad \cdots\quad&\quad \epsilon\\
\epsilon\quad&\quad1\quad&\quad\cdots\quad&\quad\epsilon\\
\vdots\quad & \quad\vdots\quad &\quad \ddots\quad &\quad \vdots\\
\epsilon\quad&\quad \epsilon\quad &\quad \cdots\quad &\quad 1
\end{bmatrix}\cdot y.
\end{align*}
Using this notation, we define a set
\begin{align*}
  \mathcal{C}_{\epsilon}:=\{y\in\mathbb{R}^m:\mathcal{T}_{\epsilon}(y)\geqq 0\}.
\end{align*}
By Definition 2.1.7 in \cite{ref58}, the set $\mathcal{C}_{\epsilon}$ is a convex polyhedral ordering cone, and thus the corresponding cone order $\leq_{\epsilon}$ on $\mathbb{R}^m$ by
\begin{align*}
  y^1\leq_{\epsilon}y^2\Longleftrightarrow y^2-y^1\in\mathcal{C}_{\epsilon}\backslash\{0\}.
\end{align*}

The following lemmas hold for $\leq_{\epsilon}$:
\begin{lemma}{\rm \cite{ref13}}\label{le:1}
For $y^1,y^2\in\mathbb{R}^m$, we have $y^1\leq_{\epsilon}y^2$ if and only if $\mathcal{T}_{\epsilon}(y^1)\leq\mathcal{T}_{\epsilon}(y^2)$.
\end{lemma}
\begin{proof}
we have
\begin{align*}
  y^1\leq_{\epsilon}y^2\Leftrightarrow y^2-y^1\in\mathcal{C}_{\epsilon}\backslash\{0\}\Leftrightarrow \mathcal{T}_{\epsilon}(y^2-y^1)\geq 0\Leftrightarrow \mathcal{T}_{\epsilon}(y^2)\geq \mathcal{T}_{\epsilon}(y^1).
\end{align*}
by the definition of $\leq_{\epsilon}$ and $\mathcal{C}_{\epsilon}$ and by the linearity of $\mathcal{T}_{\epsilon}$. \qed
\end{proof}

\begin{lemma}\label{le:2}
The cone order $\leq_{\epsilon}$ is a strict partial order.
\end{lemma}
\begin{proof}
By Definition 2.3.1 in \cite{ref58}, we would like to show that $\leq_{\epsilon}$ is irreflexive and transitive. First, it is easy to see that for all $y\in \mathbb{R}^m$, $y\nleq_{\epsilon} y$, i.e., $\leq_{\epsilon}$ is irreflexive. Next we will prove $\leq_{\epsilon}$ is transitive. For $y^1,y^2,y^3\in\mathbb{R}^m$, assume $y^1\leq_{\epsilon} y^2$ and $y^2\leq_{\epsilon} y^3$. By Lemma 1, we have
$\mathcal{T}_{\epsilon}(y^1)\leq\mathcal{T}_{\epsilon}(y^2)$ and $\mathcal{T}_{\epsilon}(y^2)\leq\mathcal{T}_{\epsilon}(y^3)$. Due to the transitivity of $\leq$, we have $\mathcal{T}_{\epsilon}(y^1)\leq\mathcal{T}_{\epsilon}(y^3)$, meaning that $y^1\leq_{\epsilon} y^3$. \qed
\end{proof}

\begin{lemma}{\rm \cite{ref22}}\label{le:3}
For $y^1,y^2\in\mathbb{R}^m$, if $y^1\leq y^2$, we have $y^1\leq_{\epsilon}y^2$.
\end{lemma}
\begin{proof}
  According to the Pareto dominance, we know $y^2-y^1\in\mathbb{R}^m_+\backslash\{0\}\subseteq\mathcal{C}_{\epsilon}\backslash\{0\}$, following $\mathcal{T}_{\epsilon}(y^2-y^1)\geq0$. By the linearity of $\mathcal{T}_{\epsilon}$ and Lemma 1, we have $y^1\leq_{\epsilon}y^2$. \qed
\end{proof}

According to the strict partial order $\leq_{\epsilon}$, the $\epsilon$-dominance relation between $x^1,x^2\in\Omega$ can be defined as follows:
\begin{align*}
x^1~\epsilon\hbox{-dominates}~x^2\Longleftrightarrow F(x^1)\leq_\epsilon F(x^2)\Longleftrightarrow\mathcal{T}_{\epsilon}(F(x^1))\leq\mathcal{T}_{\epsilon}(F(x^2)).
\end{align*}
 A nonempty set $\mathcal{U}\subseteq\mathbb{R}^m$ is called non-$\epsilon$-dominated set, if for any $y^1\in\mathcal{U}$, there does not exist $y^2\in\mathcal{U}$ such that $y^2\leq_\epsilon y^1$.

The following theorem discuss the relationship between $\epsilon$-dominance and $\epsilon$-properly Pareto optimal solutions.

\begin{theorem}\label{th:1}
   For $\epsilon\in[0,1]$, let $X^*_\epsilon$ be the $\epsilon$-properly Pareto optimal solution set of problem (2.1). Then $x\in X^*_\epsilon$ if and only if there does not exist $x'\in\Omega$, such that $F(x')\leq_\epsilon F(x)$.

\end{theorem}
\begin{proof}
    ($\Rightarrow$) Assume that $x\in X^*_\epsilon$. Suppose, to the contrary, that there exists $x'\in \Omega$, such that $F(x')\leq_\epsilon F(x)$. By Lemma 1, we have
\begin{align*}
  F(x')\leq_\epsilon F(x)\Leftrightarrow \mathcal{T}_{\epsilon}(F(x'))\leq\mathcal{T}_{\epsilon}(F(x))\Leftrightarrow F(x)-F(x')\in\mathcal{C}_{\epsilon}\backslash\{0\}.
\end{align*}
Furthermore, it is easy to prove $\mathbb{R}^m_{\epsilon}=\mathcal{C}_{\epsilon}$, thus we have $F(x)\in F(x')+\mathbb{R}^m_{\epsilon}\backslash\{0\}$, which is a contradiction to the fact that $x\in X^*_\epsilon$.

($\Leftarrow$) Assume that for $x\in\Omega$, there does not exist $x'\in\Omega$, such that $F(x')\leq_\epsilon F(x)$ and to the contrary $x\notin X^*_\epsilon$. According to Definition 2, we know that there exists $\bar{x}\in\Omega$, such that $F(x)\in F(\bar{x})+\mathcal{C}_{\epsilon}\backslash\{0\}$. By Lemma 1 and the linearity of $\mathcal{T}_{\epsilon}$, we obtain a contradiction $F(\bar{x})\leq_\epsilon F(x)$. \qed
\end{proof}

According to Theorem 1, it is easy to see that the $\epsilon$-properly Pareto optimal solution set can be found from a known Pareto optimal solution set by $\epsilon$-dominance relation. On the other hand, the multiobjective branch and bound algorithms with the Pareto dominance-based discarding test can efficiently approximate the Pareto optimal solution set, thus the Pareto dominance in the discarding test can be generalized to the $\epsilon$-dominance relation to ensure that $\epsilon$-properly Pareto optimal solutions are identified. As a result, we can derive the following discarding test:

{\bf $\epsilon$-Discarding test} \emph{Let problem (2.1) be given, let $B$ be a subbox and $l(B)$ its lower bound. For $\epsilon\in[0,1]$, if there exists a feasible objective vector $u\in F(\Omega)$, such that $u\leq_\epsilon l(B)$, then $B$ will be discarded.}

Now we state the correctness of the proposed discarding test.

\begin{theorem}\label{th:2}
Let a subbox $B\in\Omega$ and its lower bound $l(B)\in\mathbb{R}^m$ be given, let $\mathcal{U}$ be a nondominated upper bound set of problem (2.1). For $\epsilon\in[0,1]$, if there exists an upper bound $u\in\mathcal{U}$ such that $u\leq_{\epsilon} l(B)$, then $B$ does not contain $\epsilon$-properly Pareto optimal solution of problem (2.1).
\end{theorem}
\begin{proof}
Let us assume that $x^*\in B$ is an $\epsilon$-properly Pareto solution of problem (2.1). According to (3.1), we know that $l(B)\leq F(x^*)$. By Lemmas 2 and 3, we then have $u\leq_{\epsilon} l(B)\leq_{\epsilon} F(x^*)$. Thus from Theorem 1, we obtain a contradiction $x^*\notin X^*_\epsilon$.\qed
\end{proof}

Algorithm 2 gives an implementation of the $\epsilon$-discarding test, where the flag $D$ stands for decision to discard the subbox after the algorithm. Suppose that a non-$\epsilon$-dominated upper bound set $\mathcal{U}$ is known, and if there exists an upper bound $u\in\mathcal{U}$ such that $u\leq_{\epsilon} l(B)$, then the flag $D$ for $B$ is set to 1; otherwise, the flag $D$ is set to 0.

\begin{algorithm}
\caption{\texttt{$\epsilon$-DiscardingTest($B$,$l(B)$,$\mathcal{U}$)}}\label{alg:2}
    $D\leftarrow0$\;
    \eIf{there exists $u\in\mathcal{U}$ such that $u\leq_{\epsilon}l(B)$}
    {
    $D\leftarrow1$\;
    }
    {$D\leftarrow0$.}
    \Return $D$
\end{algorithm}

\subsection{The complete algorithm}

Having the $\epsilon$-discarding test, we can present the algorithm to find $\epsilon$-properly Pareto solutions of problem (2.1). The whole algorithm is given in Algorithm 3. Another algorithmic design that needs to be mentioned is that we employ a parallel breadth first search strategy \cite{ref44} to search all subboxes simultaneously in each iteration. Therefore, in line 4 we simultaneously bisect all subboxes in the box collection $\mathcal{B}_{k-1}$ perpendicularly to the direction of maximum width to construct the current collection $\mathcal{B}_{k}$. This branching produces subboxes that all have the same diameter, so in line 5 we only need to calculate the diameter of one subbox to obtain $\omega_k$. After bisection, the feasibility test mentioned in \cite{ref11} will be used to exclude the subboxes that do not contain any feasible solutions from further exploration.

In the first for-loop, the lower bound $l(B)=(l_1,\ldots,l_m)^T$ of the subbox $B$ is calculated by
\begin{align}
l_i = f_i(m(B))-\frac{L_i}{2}\|\omega(B)\|,\quad i=1,\ldots,m,\label{E:3.2}
\end{align}
where $L_i$ is the Lipschitz constant of $f_i$. Then, all lower bounds are stored in a lower bound set $\mathcal{L}$. For each lower bound $l\in \mathcal{L}$, we compare $l$ with other lower bounds stored in $\mathcal{L}$ by $\epsilon$-dominance relation: if $l$ is $\epsilon$-dominated by any other lower bounds, then $l$ will be removed from $\mathcal{L}$; otherwise, the lower bounds that are $\epsilon$-dominated by $l$ will be removed from $\mathcal{L}$. In this way, we can find a non-$\epsilon$-dominated lower bound set $\mathcal{L}_k$ from $\mathcal{L}$. The subboxes corresponding to the lower bounds stored in $\mathcal{L}_k$ constitutes $\bar{\mathcal{B}}$.

In the second for-loop, an MOEA is used to find the upper bounds for all subboxes stored in $\bar{\mathcal{B}}$. The upper bounds and the corresponding solutions (the preimage of the upper bounds) are stored in $\mathcal{U}$ and $\mathcal{X}$, respectively. In order to reduce computational costs, we apply a mini MOEA which has a small initial population size and a few generations. If the problem also has the inequality constraints, the constrained handling approaches \cite{ref15,ref46} for MOEAs can be employed to ensure that feasible solutions and upper bounds are obtained. Thereafter, the non-$\epsilon$-dominated upper bound set $\mathcal{U}_k$ and its corresponding solution set $\mathcal{X}_k$ will be found from $\mathcal{U}$ and $\mathcal{X}$ by means of $\epsilon$-dominance relation.

In the third for-loop, the discarding flags for all subboxes are calculate by Algorithm 2 and are stored into a flag list $\mathcal{D}$. Then, according to $\mathcal{D}$, in line 22 the subboxes whose flags is equal to 1 will be removed from $\mathcal{B}_k$. It is worth noting that, in order to speed up computation, we can compute the discarding flags for all the subboxes simultaneously, i.e., parallelize the third for-loop. The first and second for-loops can also be parallelized.


\begin{algorithm}
\caption{\texttt{Algorithm to find $\epsilon$-properly Pareto solutions}}\label{alg:3}
  \SetKwInOut{Input}{Input}\SetKwInOut{Output}{Output}
  \SetKwFunction{MOEA}{MOEA}
  \SetKwFunction{DT}{$\epsilon$-DiscardingTest}
  \Input{problem (2.1), $\varepsilon>0$, $\delta>0$, $\epsilon\in[0,1]$;}
    \Output{$\mathcal{B}_{k}$, $\mathcal{U}_{k}$, $\mathcal{X}_k$;}
    $k\leftarrow1$, $\mathcal{B}_0\leftarrow\Omega$, $\omega_{k-1}\leftarrow \|\omega(\Omega)\|$, $d\leftarrow10^6$\;
  \While{$d>\varepsilon$ or $\omega_{k-1}>\delta$}{
  $\mathcal{L}\leftarrow\emptyset$, $\mathcal{U}\leftarrow\emptyset$, $\mathcal{X}\leftarrow\emptyset$\;
  Construct $\mathcal{B}_{k}$ by bisecting all boxes in $\mathcal{B}_{k-1}$\;
  $\omega_k\leftarrow \max\{\|\omega(B)\|:B\in\mathcal{B}_{k}\}$\;
  Update $\mathcal{B}_{k}$ by the feasibility test suggested in \cite{ref11}\;
  \ForEach{$B\in\mathcal{B}_{k}$}
  {
    Calculate for $B$ its lower bound $l(B)$ by equation (3.2)\;
    $\mathcal{L}\leftarrow \mathcal{L}\cup l(B)$\;
  }
  Find a non-$\epsilon$-dominated lower bound set $\mathcal{L}_k$ from $\mathcal{L}$ by the $\epsilon$-dominance\;
  Determine the box collection $\bar{\mathcal{B}}\subset\mathcal{B}_k$ according to $\mathcal{L}_k$\;
  \ForEach{$B\in\bar{\mathcal{B}}$}
  {
    $\bar{\mathcal{U}},\bar{\mathcal{X}}\leftarrow\MOEA(B)$\;
    $\mathcal{U}\leftarrow\mathcal{U}\cup \bar{\mathcal{U}}$, $\mathcal{X}\leftarrow\mathcal{X}\cup \bar{\mathcal{X}}$\;
  }

  Find a non-$\epsilon$-dominated upper bound set $\mathcal{U}_k$ and the corresponding solution set $\mathcal{X}_k$ by the $\epsilon$-dominance\;
  \ForEach{$B\in\mathcal{B}_k$}
  { $D(B)\leftarrow \DT(B,l(B),\mathcal{U}_k)$\;
  $\mathcal{D}\leftarrow \mathcal{D}\cup D(B)$\;
  }
  Update $\mathcal{B}_k$ according to the flag list $\mathcal{D}$\;
  $d\leftarrow d_h(\mathcal{U}_{k},\mathcal{L}_{k})$, $k\leftarrow k+1$.
  }
\end{algorithm}

In order to ensure the trade-offs between disparately scaled objectives can be correctly expressed, we use an objective normalization technique:
\begin{align*}
\bar{f}_i=\frac{f_i-z^*_i}{z^{nad}_i-z^*_i},
\end{align*}
where $z^{nad}=(z^{nad}_1,\ldots,z^{nad}_m)^T$ and $z^*=(z^*_1,\ldots,z^*_m)^T$ are the nadir and ideal points, respectively. There are two ways to obtain the nadir and ideal points: one is to solve each objective as a single-objective optimization problem to obtain the extreme points of the Pareto front. In this case, we should choose a global algorithm to solve the single-objective optimization problems; the second approach is to update $z^{nad}$ and $z^*$ during the solution process. The initial values of $z^{nad}$ and $z^*$ can be obtained from the natural interval extension of the objectives. The lower bound of the natural interval expansion is $z^*$, and the upper bound is $z^{nad}$. During the iterations, the minimum value of $f_i$ in the current lower bounds is used to update $z^*_i$, while the maximum value of $f_i$ in the current upper bounds is used to update $\tilde{z}^{nad}$. If a subbox can search for the maximum value of $f_i$ in the upper bounds or the minimum value of $f_i$ in the lower bounds, then it will be stored in the current box collection and not be removed by the $\epsilon$-discarding test.

\subsection{Convergence results}
We start by showing the termination of the algorithm.
\begin{theorem}\label{th:3}
Let the predefined parameters $\varepsilon>0$ and $\delta>0$ be given, Algorithm \ref{alg:3} terminates after a finite number of iterations.
\end{theorem}
\begin{proof}
Because we divide all boxes perpendicular to a side with maximal width, $\omega_k$ decreases among the sequence of box collections, i.e., $\omega_{k} > \omega_{k+1}$ for every $k$ and converges to 0. Therefore, for a given $\delta>0$, there must exist a iteration count $\tilde{k}>0$ such that $\omega_{\tilde{k}}\leq\delta$.

Assume we use (3.2) to calculate lower bounds. According to the way $\mathcal{U}_{k}$ is constructed and the $\epsilon$-discarding test, for every $u\in\mathcal{U}_{k}$, there exists a subbox $B\in\mathcal{B}_k$, such that $F^{-1}(u)\in B$ and $l(B)\in\mathcal{L}_k$. Then, based on the Lipschitz condition, we have
\begin{align*}
d(u,\mathcal{L}_k)\leq d(u,l(B))= \|u-F(m(B)+\frac{L}{2}\omega(B))\|&\leq\|u-F(m(B)\|+\frac{1}{2}\omega_k\|L\|\\
&\leq\omega_k\|L\|
\end{align*}
where $L=(L_1,\ldots,L_m)^T$ consisting of the Lipschitz constants of objectives. Hence we have
\begin{align}
d_h(\mathcal{U}_{k},\mathcal{L}_{k}) = \omega_k\|L\|.\label{E:3.3}
\end{align}
Due to the fact that $\omega_k$ converges to 0, it follows that for a given $\varepsilon>0$, there must exist a iteration count $\bar{k}>0$ such that $d_h(\mathcal{U}_{k},\mathcal{L}_{k})\leq\varepsilon$.\qed
\end{proof}

The next theorem states all properly Pareto optimal solutions of problem (2.1) are contained in the union of subboxes generated by the algorithm.
\begin{theorem}\label{th:4}
Let $X_\epsilon^*$ be an $\epsilon$-properly Pareto solution set of problem (2.1) and $\{\mathcal{B}_k\}_{k\in\mathbb{N}}$ a sequence of box collections generated by Algorithm 3. Then, for arbitrary $k\in\mathbb{N}$ we have $X_\epsilon^*\subseteq\bigcup_{B\in\mathcal{B}_k} B$.
\end{theorem}
\begin{proof}
Let us assume that there exists $k\in\mathbb{N}$ and an $\epsilon$-properly Pareto solution $\bar{x}^*\in X_\epsilon^*$ such that $\bar{x}^*\notin\bigcup_{B\in\mathcal{B}_k} B$, meaning that $\bar{x}^*$ is contained in a removed subbox $B$ in the pervious iteration. From the $\epsilon$-discarding test, we know that $B$ will be discarded if and only if there exists a feasible objective vector $u$ such that $u\leq_{\epsilon} l(B)$. Due to the lower bound (3.2) and Lemma 2, we then have $l(B)\leq_{\epsilon} F(\bar{x}^*)$. Thus we know that $u\leq_{\epsilon} F(\bar{x}^*)$, which contradicts the assumption that $\bar{x}^*$ is an $\epsilon$-properly Pareto solution of problem (2.1).\qed
\end{proof}

\begin{corollary}
Let $X_\epsilon^*$ be an $\epsilon$-properly Pareto solution set of problem (2.1) and $\{\mathcal{B}_k\}_{k\in\mathbb{N}}$ a sequence of box collections generated by Algorithm 3. For each $x^*\in X_\epsilon^*$ and $k\in\mathbb{N}$, these exists $B(x^*,k)\in\mathcal{B}_k$, such that $x^*\in B(x^*,k)$. Furthermore, we have $\lim\limits_{k\rightarrow \infty}d_H(X_\epsilon^*,\bigcup\limits_{x\in X_\epsilon^*}B(x,k))=0$.
\end{corollary}
\begin{proof}
The first conclusion is guaranteed by Theorem 4. Next we would like to prove the second conclusion.

On the one hand, from the first conclusion, we have
\begin{align*}
d(x^*,\bigcup\limits_{x\in X_\epsilon^*}B(x,k)) =0, \quad x^*\in X^*_\epsilon,~k\in\mathbb{N},
\end{align*}
it follows $d_h(X^*_\epsilon,\bigcup\limits_{x\in X_\epsilon^*}B(x,k))=0$. On the other hand, for each $x\in\bigcup\limits_{x\in X_\epsilon^*}B(x,k)$, there exists $B(\hat{x}^*,k)\in\bigcup\limits_{x\in X_\epsilon^*}B(x,k)$, such that $x\in B(\hat{x}^*,k)$. We then have
\begin{align*}
0\leq\lim\limits_{k\rightarrow \infty}d(x,X_\epsilon^*)\leq\lim\limits_{k\rightarrow \infty}d(x,\hat{x}^*)\leq\lim\limits_{k\rightarrow \infty} w_k =0,
\end{align*}
meaning that $\lim\limits_{k\rightarrow \infty}d_h(\bigcup\limits_{x\in X_\epsilon^*}B(x,k), X_\epsilon^*)=0$. The second conclusion is proven. \qed
\end{proof}

In the next theorem, we prove that the solution set generated by Algorithm 3 is an $\varepsilon$-efficient solution set of problem (2.1).

\begin{theorem}
Assume that the midpoints of the subboxes are used to calculate the upper bounds in Algorithm 3. Let $\mathcal{X}$ be the solution set generated by Algorithm 3 and $\mathcal{L}_k$ the non-$\epsilon$-dominated lower bound set. Then $\mathcal{X}$ is an $\varepsilon$-efficient solution set of problem (2.1).
\end{theorem}
\begin{proof}
  Suppose $\tilde{x}\in \mathcal{X}$ is the midpoint of the subbox $B\in\mathcal{B}_k$ and $l\in\mathcal{L}_k$ is corresponding lower bound. According to (3.3), we know that
  \begin{align}
    \varepsilon\geq \frac{1}{2}\omega_k\|L\|>\frac{1}{2}\omega_k L_{{\rm max}},\label{IE:3.4}
  \end{align}
  where $L_{max}=\max\{L_i,i=1,\dots,m\}$. Then we can obtain a lower bound $\tilde{l}=(\tilde{l}_1,\ldots,\tilde{l}_m)^T$ whose component can be calculated by
  \begin{align*}
    \tilde{l}_i=f(\tilde{x})_i-\frac{1}{2}\omega_k L_{{\rm max}},
  \end{align*}
  and further, it is easy to see that $F(\tilde{x})-\varepsilon e <\tilde{l}\leqq l$.

  In the following we will prove $\tilde{x}$ is an $\varepsilon$-efficient solution of problem (2.1) in two aspects. On the one hand, by (3.1), we have
  \begin{align*}
    F(\tilde{x})-\varepsilon e <\tilde{l}\leqq l\leq F(x),\quad x\in B.
  \end{align*}
  Therefore, there does not exist $x\in B$ with $F(x)\leq F(\tilde{x})-\varepsilon e$.

  On the other hand, assume there exists another subbox $B'\in\mathcal{B}_k\backslash B$ and a feasible point $x'\in B'$ such that $F(x')\leq F(\tilde{x})-\varepsilon e$. By (3.2), (3.4) and Lemma 3, we have
    \begin{align*}
    F(x')-\frac{1}{2}\omega_k\|L\|\leq F(x')\leq F(\tilde{x})-\varepsilon e<\tilde{l}\leqq l\Longleftrightarrow F(x')-\frac{1}{2}\omega_k\|L\|\leq_\epsilon l,
  \end{align*}
  which is a contradiction to the fact $\mathcal{L}_k$ is a non-$\epsilon$-dominated lower bound set. Therefore, there does not exist a subbox $B'\in\mathcal{B}_k\backslash B$ and a point $x'\in B'$ with $F(x')\leq F(\tilde{x})-\varepsilon e$. Now, we prove
  the conclusion. \qed
\end{proof}

To make it easier to obtain the conclusion in Theorem 3, we consider only the situation where the upper bounds are computed by using the midpoints, rather than mini MOEA. However, in practice, the tightness of the upper bounds obtained by mini MOEA tends to be better than  the ones calculated by midpoints. Therefore, the accuracy of the solution set generated by Algorithm 3 will be higher than the accuracy of the one discussed in Theorem 5.

\section{Experimental Results}

Algorithm 3 is implemented in Python 3.8 with fundamental packages like numpy, scipy and multiprocessing, and  performs on a computer with Intel(R) Core(TM) i7-10700 CPU and 32 Gbytes RAM on operation system WINDOWS 10 PROFESSIONAL. In all experiments, we set $\epsilon=0.75$ to search for the $\epsilon$-properly Pareto solutions. For the MOEA employed in Algorithm 3, we use MOEA/D-DE \cite{ref20} with the population size 10 and 20 generations.

It is worth noting that we do not intend to compare Algorithm 3 with other algorithms. This is because the aim of most branch and bound algorithms is to search for Pareto optimal solution set rather than $\epsilon$-properly Pareto solution set. Moreover, although some scalarization-based methods can obtain $\epsilon$-properly Pareto solutions, not only they are not easily applied to some nonconvex problems, but also . Therefore, in the following we only demonstrate the effectiveness of Algorithm 3 on some test problems as well as engineering constrained optimization problems.

\begin{example}
  First, we consider three test problems.
  \begin{itemize}
    \item {\rm MOP} \cite{ref37}
        $$F(x)=
        \begin{pmatrix}
        0.5(\sqrt{1+(x_1+x_2)^2}+\sqrt{1+(x_1-x_2)^2}+x_1-x_2)+e^{-(x_1-x_2)^2}\\
        0.5(\sqrt{1+(x_1+x_2)^2}+\sqrt{1+(x_1-x_2)^2}-x_1+x_2)+e^{-(x_1-x_2)^2}
        \end{pmatrix},$$
    where $x_i\in[-3,3],\;i=1,2$. This problem has a disconnected Pareto front, which has two knees. We use $(\varepsilon,\delta)=(0.001,0.0001)$.
  \item {\rm DEB2DK} \cite{ref3}
      $$F(x)=
        \begin{pmatrix}
        g(x)r(x_1)\sin(0.5\pi x_1)\\
        g(x)r(x_1)\cos(0.5\pi x_1)
        \end{pmatrix},$$
    where
        \begin{align*}
        &g(x)=1+\frac{9}{n-1}\sum_{i=2}^{n}x_i,\\
        &r(x_1)=5+10(x_1-0.5)^2+\frac{1}{K}\cos(2K\pi x_1),\\
        &x_i\in[0,1],\;i=1,2.
        \end{align*}
        The parameter $K$ allows to control the number of knees in DEB2DK, and we set $K=4$, $n=3$ and $(\varepsilon,\delta)=(0.0015,0.00015)$.
  \item {\rm DEB3DK} \cite{ref3}
        $$F(x)=
        \begin{pmatrix}
        &g(x)r(x_1,x_2)\sin(0.5\pi x_1)\sin(0.5\pi x_2)\\
        &g(x)r(x_1,x_2)\sin(0.5\pi x_1)\cos(0.5\pi x_2)\\
        &g(x)r(x_1,x_2)\cos(0.5\pi x_1)
        \end{pmatrix},$$
        where
        \begin{align*}
        &g(x)=1+\frac{9}{n-1}\sum_{i=3}^{n}x_i,\\
        &r(x_1,x_2) = (r_1(x_1)+r_2(x_2))/2\\
        &r_i(x_i)=5+10(x_i-0.5)^2+\frac{1}{K}\cos(2K\pi x_i),\\
        &x_i\in[0,1],\;i=1,2.
        \end{align*}
        In DEB3DK, we set $K=1$, $n=3$ and $(\varepsilon,\delta)=(0.006,0.008)$.
  \end{itemize}
\end{example}

The experimental results for the three test problems are shown in Fig. 2. The Pareto fronts (blue dots) of three problems are found by Algorithm 3 with $\epsilon=0$, and the red stars are the upper bounds obtained by Algorithm 3 with $\epsilon=0.75$. An interesting observation can be made from the figure: the upper bounds obtained by Algorithm 3 usually lie in the ``bulge'' regions on Pareto fronts in Fig.~2. Based on this, we speculate that Algorithm 3 searches for the knees in the convex regions on the Pareto fronts. From the view of trade-offs, the convex knees on Pareto front are characterized by the fact that a small improvement in one objective will cause a large deterioration in the other objective \cite{ref3}. It is easy to see from Fig.~2 that $\epsilon$-properly Pareto solutions found by Algorithm 3 also satisfy this feature. Existing knee-oriented approaches need to propose different indicators \cite{ref1,ref3,ref5,ref28} to identify knees, but these indicators are not actually mathematical definitions of knees. In contrast, $\epsilon$-properly Pareto solutions not only satisfies the geometrical characterization of the knee, but also has a clear definition. Furthermore, several MOEAs \cite{ref22,ref29} also use $\epsilon$-dominance relation or some similar dominance relation to find knees, but they do not discuss the relationship between the $\epsilon$-dominance relation and $\epsilon$-properly Pareto solutions, and thus there is no further proof of global convergence.

\begin{figure}[htbp]
\centering
\subfigure[MOP]{
\includegraphics[width=0.3\textwidth]{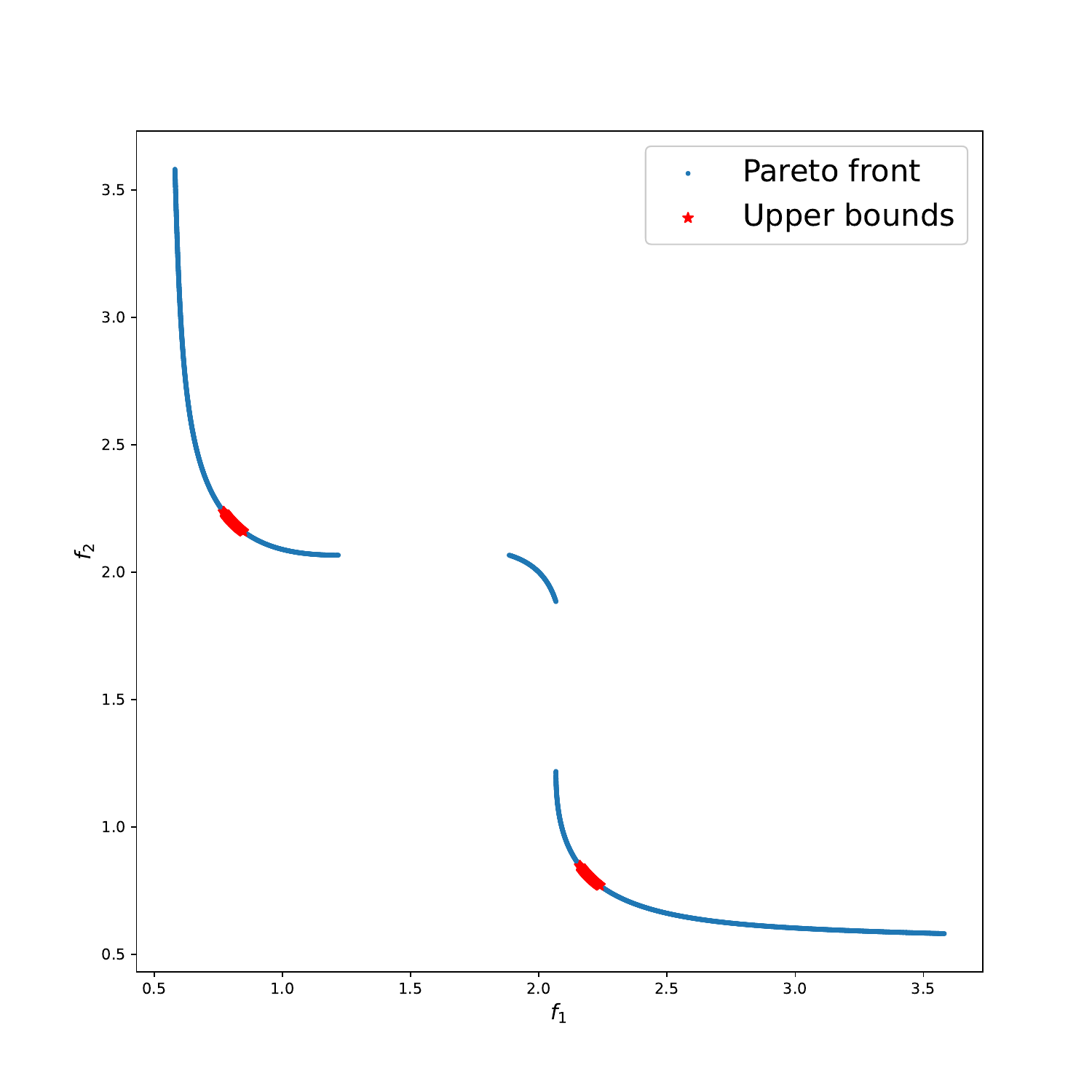}}
\subfigure[DEB2DK]{
\includegraphics[width=0.3\textwidth]{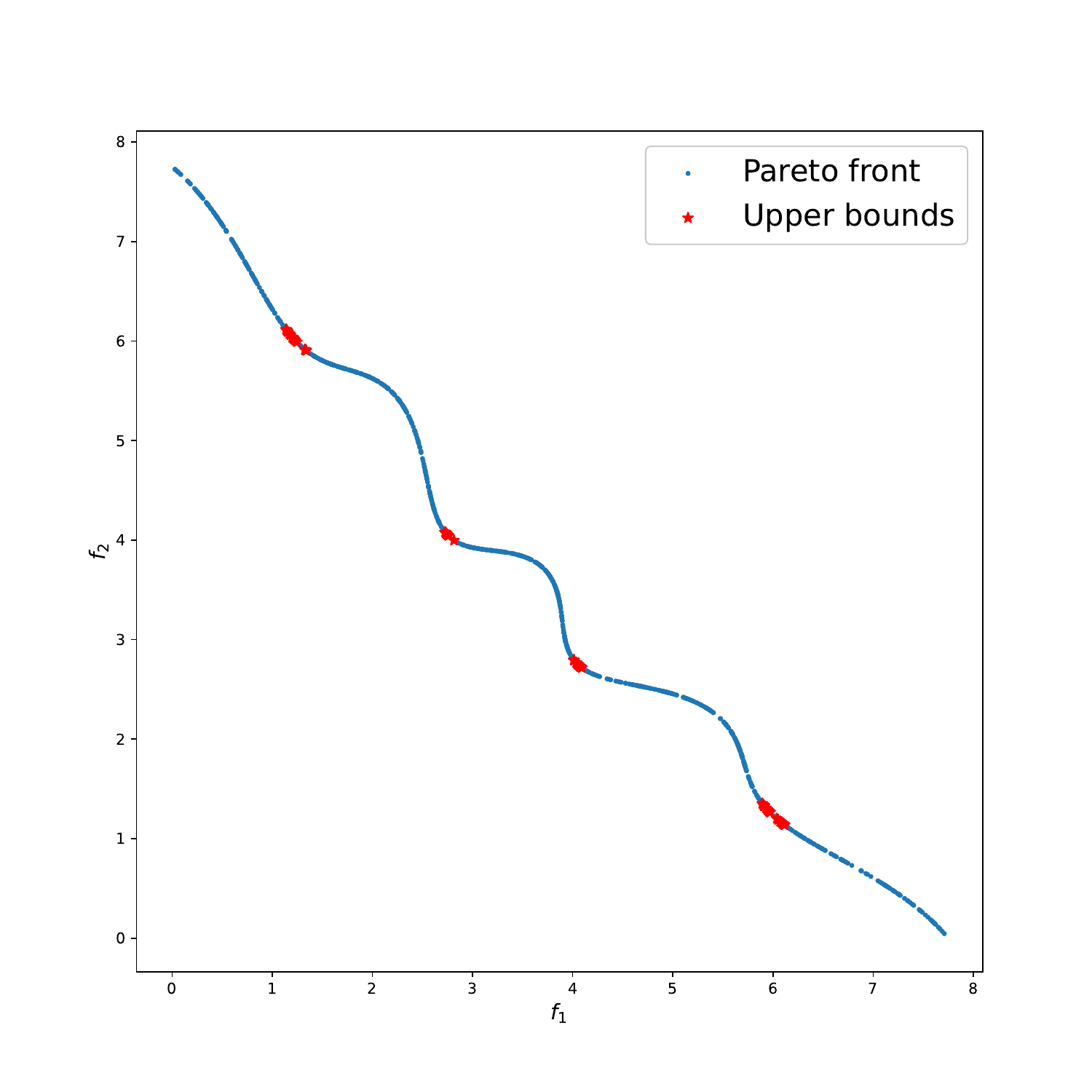}}
\subfigure[DEB3DK]{
\includegraphics[width=0.33\textwidth]{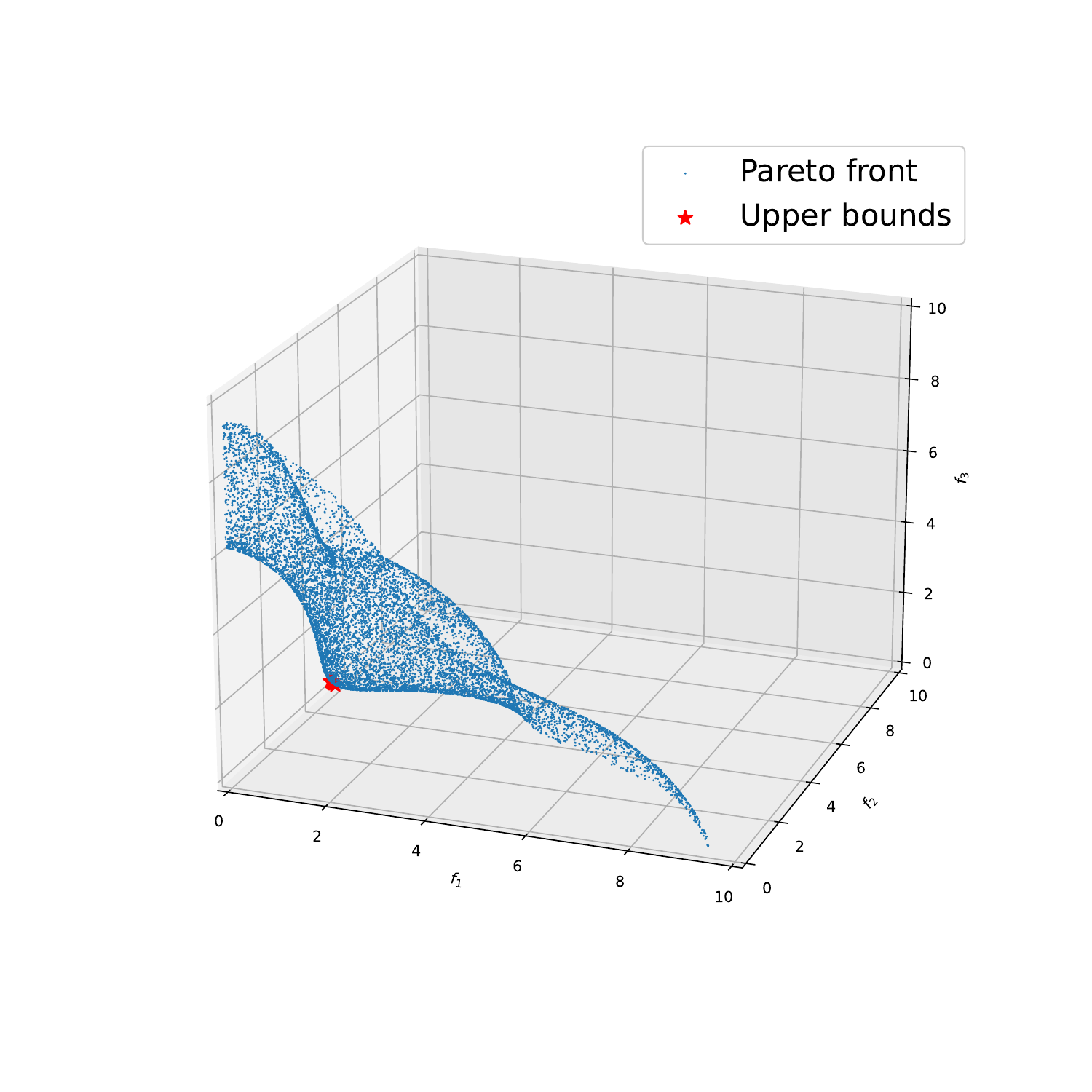}}
\caption{Results of Algorithm 3 on three test problems.}
\end{figure}


\begin{example}
  The welded beam design problem \cite{ref7,ref31} has four real-parameter variables $x=(x_1,x_2,x_3,x_4)$ and four non-linear constraints. Let us consider a manufacturing process in which minimization of the cost and minimization of the end deflection. The optimization problem is given as follows:
$$F(x)=
\begin{pmatrix}
1.10471x_1^2 x_3+ 0.04811x_2x_4(14.0+x_3)\\
2.1592/(x_2x_4^3)
\end{pmatrix}$$
subject to the constraints
\begin{align*}
&g_1(x)=13600-\tau(x)\geq0,\\
&g_2(x)=30000-\sigma(x)\geq0,\\
&g_3(x)=x_2-x_1\geq0,\\
&g_4(x)=P_c(x)-6000\geq0,\\
&0.125\leq x_1,x_2\leq 5,\\
&0.1\leq x_3,x_4\leq 10.
\end{align*}
where the stress and buckling terms are non-linear to design variables and are given as follows
\begin{align*}
&\tau(x)=\sqrt{(\tau')^2+(\tau'')^2+(x_3\tau'\tau'')/\sqrt{0.25(x_3^2+(x_1+x_4)^2)}},\\
&\tau'=\frac{6000}{\sqrt{2}x_1x_3},\\
&\tau''=\frac{6000(14+0.5x_3)\sqrt{0.25(x_3^2+(x_1+x_4)^2)}}{1.414x_1x_3(x_3^2/12+0.25(x_1+x_4)^2)},\\
&\sigma(x)=\frac{504000}{x_2x_4^2},\\
&P_c(x)=64746.022(1-0.0282346x_4)x_4x_2^3.
\end{align*}
\end{example}

Fig. \ref{fig3} shows the results on the welded beam design problem. In this problem, we set $(\varepsilon,\delta)=(0.3,0.02)$. The Pareto front are found by Algorithm 3 with $\epsilon=0$. It can be seen that most of the objective vectors on the Pareto front have very high or very low trade-offs, meaning that these solutions do not differ from the weakly Pareto optimal solutions. Therefore, if the decision maker provides preferences without sufficient a priori information, he/she is likely to obtain undesired solutions with bad trade-offs. For example, we use the reference point-based branch and bound algorithm (RBB) mentioned in the literature \cite{ref44} to solve this problem. In RBB, we use the same $\varepsilon$ and $\delta$, and assume that the decision maker provides three reference points as a priori information: (i) (4, 0.003), (ii) (20, 0.002), (iii) (32, 0.0007). The results of RBB are presented in Fig. 3b. It is not difficult to see that if the decision maker does not have a high demand for machining accuracy of the welded beam, they may not be satisfied with the solution corresponding to the second or third reference point, as he/she must pay high costs of fabrication to minimize the deflection. Therefore, the decision maker may adjust their preferences to the neighborhood of the first reference point. In contrast, Algorithm 3 directly provides the decision maker with $\epsilon$-properly Pareto optimal solutions (red stars) without any a priori information. As can be seen in Fig. 3a, these solutions take into account both the cost and the deflection, and thus are more acceptable to the decision maker.

\begin{figure}[htbp]
\centering
\subfigure[Algorithm 3]{
\includegraphics[width=0.45\textwidth]{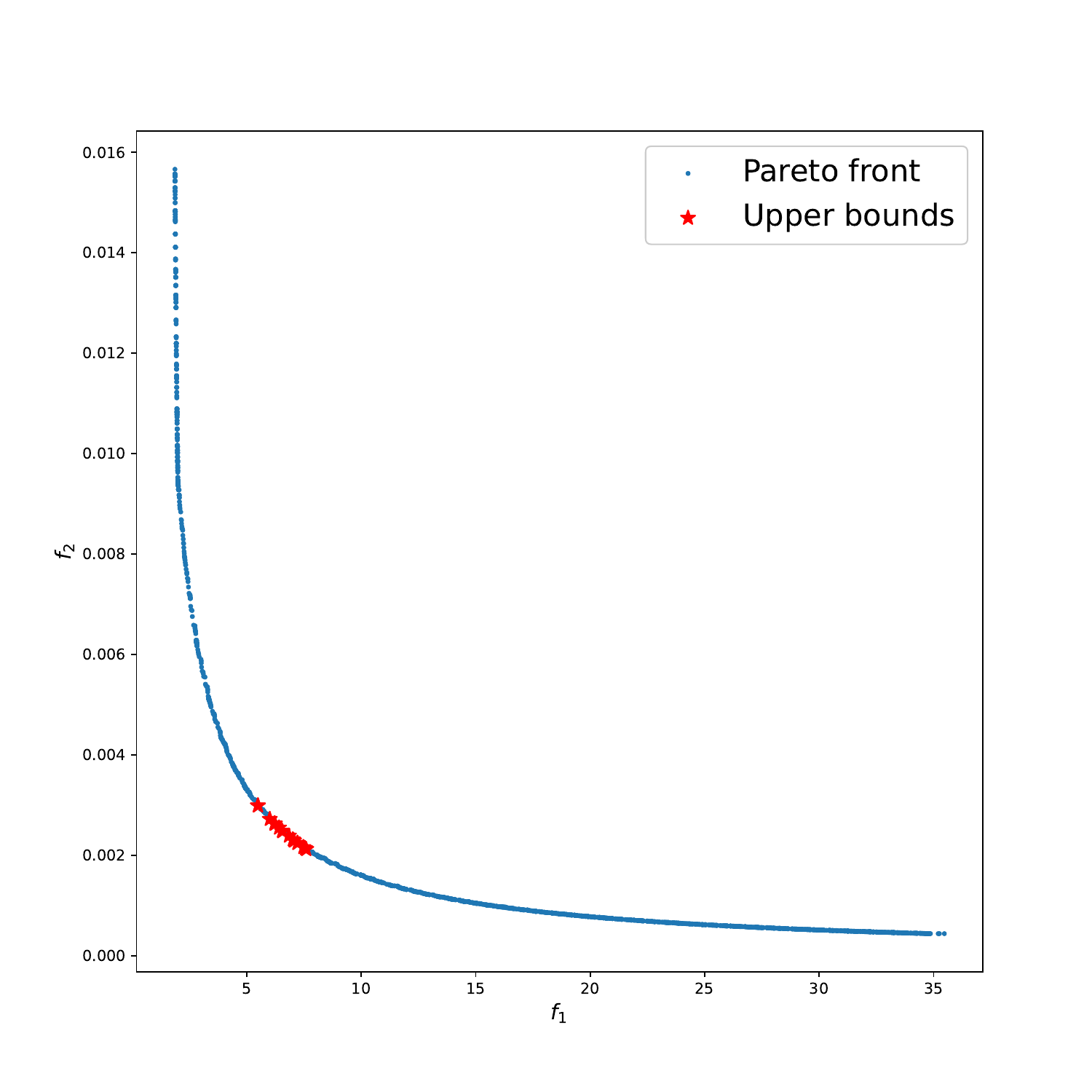}}
\subfigure[RBB\cite{ref44}]{
\includegraphics[width=0.45\textwidth]{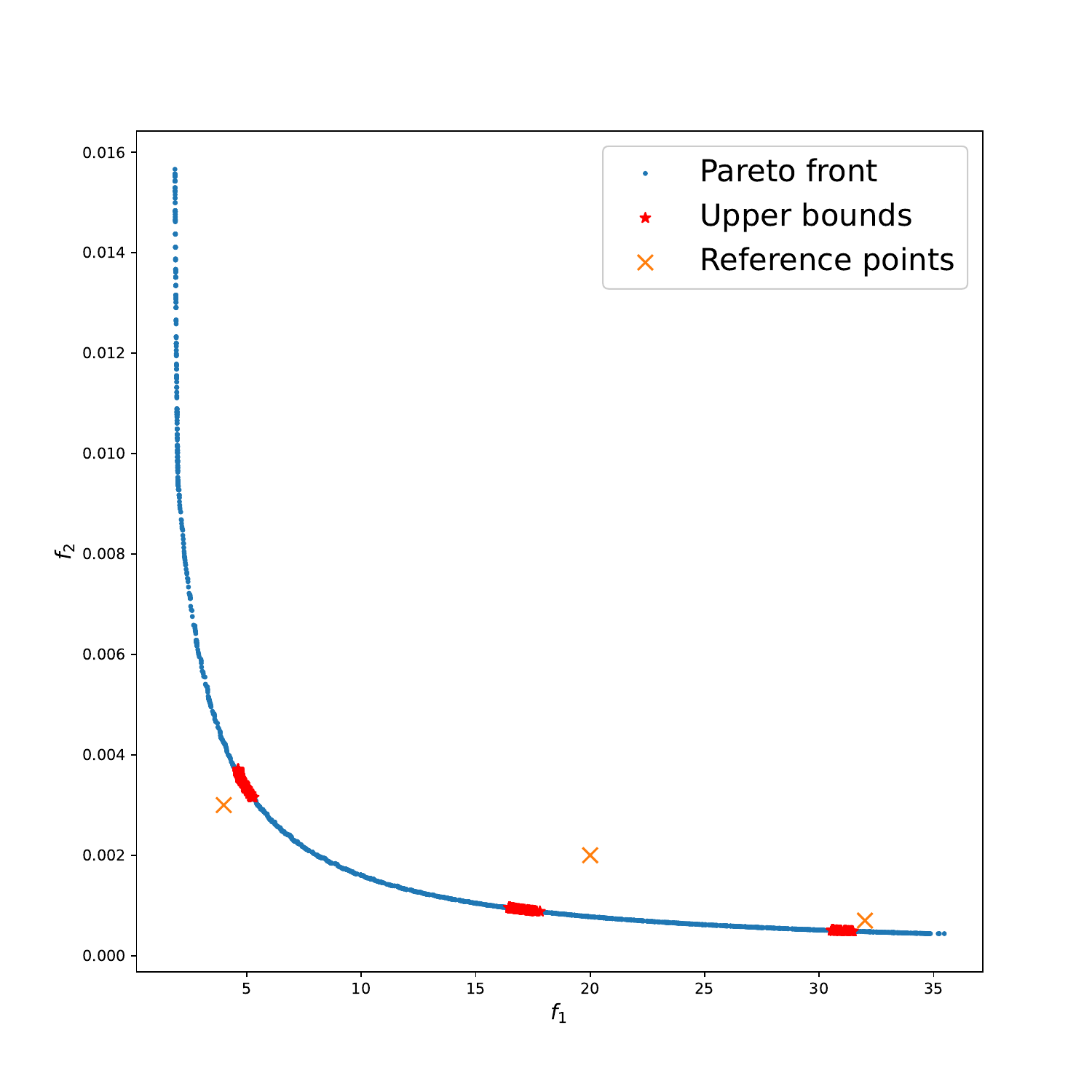}}
\caption{Result on the welded beam design problem.}
\end{figure}

\begin{example}
  The water resource planning problem \cite{ref30} involves optimal planning for a storm drainage system in an urban area. The objectives to be minimized are drainage network cost, storage facility cost, treatment facility cost, expected flood damage cost and expected economic loss due to flood.  The detailed description of the problem and constraints can be obtained from \cite{ref24}:
$$F(x)=
\begin{pmatrix}
 106780.37(x_2 + x_3)+61704.67\\
3000x_1\\
30570*0.02289.0x_2/(0.06*2289.0)^{0.65}\\
250.0*2289.0\exp(-39.75x_2+9.9x_3+2.74)\\
25.0((1.39/(x_1x_2)) + 4940.0x_3-80.0)
\end{pmatrix}$$
subject to the constraints
\begin{align*}
g_1(x)&=0.00139/(x_1x_2)+4.94x3-0.08\leq1,\\
g_2(x)&=0.000306/(x_1x_2)+1.082x_3-0.0986\leq1,\\
g_3(x)&=12.307/(x_1x_2)+49408.24x_3+4051.02\leq50000,\\
g_4(x)&=2.098/(x_1x_2)+8046.33x-696.71\leq16000,\\
g_5(x)&=2.138/(x_1x_2)+7883.39x_3-705.04\leq10000,\\
g_6(x)&=0.417(x_1x_2)+1721.26x3-136.54\leq2000,\\
g_7(x)&=0.164/(x_1x_2)+631.13x3-54.48\leq550,\\
0.01&\leq x_1\leq 0.45,~0.01\leq x_2,x_3\leq 0.1.
\end{align*}
\end{example}

\begin{figure}[htbp]%
\centering
\includegraphics[width=0.9\textwidth]{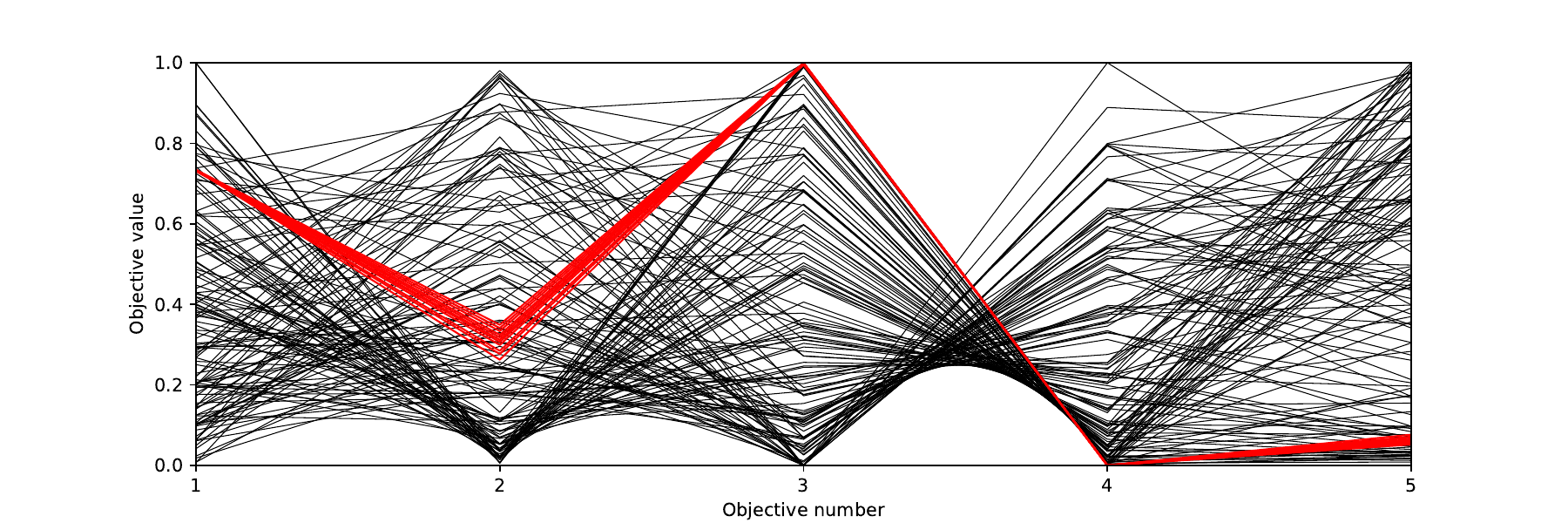}
\caption{Results on the water resource planning problem}
\end{figure}


In this problem, we set $(\varepsilon,\delta)=(0.1,0.02)$. The results of Algorithm 3 on the water resource planning problem are depicted in Fig. 4. The black lines are the value paths of the Pareto front obtained by Algorithm 3 with $\epsilon=0$. Algorithm 3 with $\epsilon=0$ obtains a large number of candidates to represent the high-dimensional Pareto front, we only select some of them as representative solutions. The value paths of the images of $\epsilon$-properly Pareto optimal solutions obtained by Algorithm 3 are plotted in red lines. It can be seen that the solutions provided by Algorithm 3 are only a small part of the entire Pareto front, so the decision maker does not confront a high level of decision pressure.

\section{Conclusion}

Many branch and bound algorithms for MOPs aim to approximate the whole Pareto optimal solution set. Their solution processes are considered resource-intensive and time-consuming. In particular, if the number of objectives is large, it is very difficult for the decision maker to select the most interesting solution from a large number of candidate solutions. Although the reference point-based branch and bound algorithm \cite{ref44} can obtain the regions of interest that match the decision maker's preferences, it is not feasible to require the decision maker to explicitly specify a priori preferences in some situation. As a result, these algorithms may not be easy to use for decision makers.

In this paper, we have argued that, without the a priori preference, the properly Pareto optimal solutions are likely to be the most relevant to the decision maker. Consequently, we have proposed a new branch and bound algorithm to find so-called $\epsilon$-properly Pareto optimal solutions. The basic idea was to replace the Pareto dominance relation in the discarding test with the $\epsilon$-dominance relation that is induced by a convex polyhedral cone. In this way, the subboxes which do not contain the $\epsilon$-properly Pareto optimal solution will be removed, resulting in a significant reduction in the number of candidate solutions. We prove that our algorithm is able to obtain a set of approximate solutions in a finite iteration. Numerical experiments confirm the effectiveness and application value of the proposed algorithm.

We are currently working on a refined version of the proposed algorithm, which allows to control the distribution of properly Pareto optimal solutions according to the skews of the Pareto front. Furthermore, it would be interesting to propose a branch and bound algorithm for vector optimization problems, and to test it on some real-world problems.


\begin{acknowledgements}
This work is supported by the Major Program of National Natural Science Foundation of China (Nos. 11991020, 11991024), the General Program of National Natural Science Foundation of China (No. 11971084), the Team Project of Innovation Leading Talent in Chongqing (No. CQYC20210309536) and the NSFC-RGC (Hong Kong) Joint Research Program (No. 12261160365).
\end{acknowledgements}

\end{document}